\documentclass[12pt]{article}
\usepackage{graphicx}
 \usepackage[reqno, namelimits, sumlimits]{amsmath}
   \usepackage{amsfonts}
   \usepackage{amssymb}
  \usepackage{amsthm}
\numberwithin{equation}{section}
{\theoremstyle{definition}
\newtheorem{definition}{Definition}[section]

\newtheorem{remark}[definition]{Remark}
}
\newtheorem{proposition}[definition]{Proposition}
\newtheorem{corollary}[definition]{Corollary}
\newtheorem{lemma}[definition]{Lemma}
\newtheorem{thm}[definition]{Theorem}
\renewenvironment{definition}{\begin{definition}}{\mbox{}
\end{definition}}

\newcommand{\R}{\mathbb{R}}

\def\R{{\mathbb R}}

\MakeLowercase{}

\font\textmsbm=msbm10
\newfam\mathbbfam
\textfont\mathbbfam=\textmsbm

\catcode`\Ž=13
\def Ž{\'e}
\catcode`\=13
\def {\`e}
\catcode`\ˆ=13
\def ˆ{\`a}
\catcode`\‰=13
\def ‰{\^a}
\catcode`\=13
\def {\`u}
\catcode`\ž=13
\def ž{\^u}
\catcode`\=13
\def {\c{c}}
\catcode`\™=13
\def ™{\^o}
\catcode`\=13
\def {\^e}
\catcode`\"=13
\def "{\^i}

\date{ }

\title{On the ultimate energy bound of solutions to some forced second order evolution equations with a general nonlinear damping operator.}
\author{Alain Haraux\vspace{1ex}\\ 
{\normalsize Sorbonne Universit\'es, UPMC Univ Paris 06 \& CNRS } \\  
{\normalsize UMR 7598, Laboratoire Jacques-Louis Lions }\\ 
{\normalsize 4, place Jussieu 75005, Paris, France}\\  
{\normalsize e-mail: \texttt{haraux@ann.jussieu.fr}}}

\begin{document}
\maketitle
\begin{abstract} Under suitable growth and coercivity conditions on the nonlinear damping operator $g$ which ensure non-resonance, we estimate  the ultimate bound of the energy of the general solution  to the equation $\ddot{u}(t) + Au(t) + g(\dot{u}(t))=h(t),\quad t\in\R^+ ,$ where $A$ is a positive selfadjoint operator on a Hilbert space $H$ and $h$ is a bounded forcing term with values in $H$. In general the bound is of the form $ C(1+ ||h||^4)$ where $||h||$ stands for the $L^\infty$ norm of $h$ with values in $H$ and the growth of $g$ does not seem to play any role.  If $g$ behaves lie a power for large values of the velocity, the ultimate bound has a quadratic growth with respect to $||h||$ and this result is optimal. If $h$ is anti periodic, we obtain a much lower growth bound and again the result is shown to be optimal even for scalar ODEs. 
 \end{abstract}

 {\small \bf AMS classification numbers:} 34A34, 34D20, 35B40, 35L10, 35L90


\bigskip

\newpage
\section[Introduction et r\'esultats principaux]{Introduction.} In this paper we investigate a specific quantitative aspect of  solutions to the equation  \begin{equation*}\label{basic-eq}
\ddot{u}(t) + Au(t) + g(\dot{u}(t)) =h(t)
\end{equation*} where $V$ is a real Hilbert space, $A\in L(V, V')$ is a symmetric, positive, coercive operator, $g\in C(V, V')$ is monotone and $h$ is a forcing term. This equation has been intensively studied in the Literature when $g$ is a local damping term, covering the following topics: existence of almost periodic solutions, asymptotic behavior of the general solution, rate of decay to $0$ of the difference of two solutions in the energy space in the best cases,  cf.e.g.  \cite{AP, Prouse, Bi, B-H, H0, HDL, HR, H3, H2,  HZ, Z1}. In the more recent paper \cite{ABH^2}, a very general result generalizing the theorems of \cite{H2} on boundedness and compactness has been proved for possibly nonlocal damping terms. However when looking at the arguments of \cite{ABH^2} and \cite{H2} and trying to extract an estimate of the solutions for $t$ large,  we find immediately that the methodology is not adapted to that purpose. The present article aims at improving the situation. Actually we devise a new technique which allows to ``forget" the influence of the initial data from the very beginning of the estimates, thus dropping all unnecessary terms related to transient behavior.
The plan of the paper is as follows: In Section 2, we introduce the basic tools used in the statements and proofs of the main results. Section 3 is devoted to a very general case. Section 4 covers a still rather general case where the damping operator behaves like a power for large values of the velocity, this for instance allows to encompass any polynomial map, and we give a short list of examples in the field of PDEs  of the second order in $t$ for which our result is optimal. In section 5 we establish  two partial results when the forcing is anti-periodic, a situation which is known (cf. e.g. \cite{H4}) to prevent resonance under weaker  conditions on $g$ than the general periodic case. We obtain a better estimate which is optimal in finite dimensions, but in the infinite dimensional  we can only slightly improve the general estimate and we do not reach what one might expect to be the optimal result.  

\section{Functional framework and the initial value problem.}
In this section, we  recall the exact functional framework that shall be used in the formulation as well as in the proofs of our new results. We follow the presentation from \cite{ABH^2} at the exception of a small difference for the approximation of weak solutions.  
\subsection {Monotone operators} Let $\cal{H}$ be a real Hilbert space endowed with an inner product $(.,.)_ {\cal{H}}$. We recall that a map $\cal{A}$ defined on a subset $ \cal{D} = D(\cal{A})$ with values in $\cal{H}$ is monotone if 
$$ \forall (U, \hat{U})\in {\cal{D}} \times {\cal{D} }, \quad ({\cal{A}}U - {\cal{A}}\hat{U}, U-\hat{U})_ {\cal{H}}\ge 0.$$ In addition $\cal{A}$ is called maximal monotone if 
$$ \forall F\in {\cal{H}}, \quad  \exists U\in D(\cal{A})\quad {\cal{A}}U + U = F.$$ The following result  is well-known (cf. H. Brezis \cite{BR1}) .
\begin{proposition}\label{eveq} If $\cal{A}$ is maximal monotone, for each $T>0$, each $U_0\in D(\cal{A})$ and $F= F(t)\in W^{1,1}(0,T;{\cal{H}})$ there is a unique function 
$U \in W^{1,1}(0,T;{\cal{H}})$ with  $ U(t) \in D({\cal{A}}) )$ for almost all $t\in (0,T)$ , $U(0) = U_0$ and such that for almost all $t\in (0,T)$
\begin{equation}\label{Eveq} U'(t)+{\cal{A}}U(t) = F(t). \end{equation} In addition if  for some $\hat{U}_0\in D(\cal{A})$ and $\hat{H}\in W^{1,1}(0,T;{\cal{H}})$ we consider the solution $\hat{U} \in W^{1,1}(0,T;{\cal{H}})$ with  $ \hat{U}(t) \in D({\cal{A}}) )$ for almost all $t\in (0,T)$ , $\hat{U}(0) = \hat{U}_0$ of $$ \hat{U}'(t)+{\cal{A}}\hat{U}(t) = \hat{F}(t), $$ then the difference satisfies the inequality
$$ \forall t\in [0, T], \quad \vert U(t)- \hat{U}(t)\vert \le \quad \vert U_0- \hat{U}_0\vert + \int _0^t \vert F(s)- \hat{F}(s)\vert ds. $$
\end{proposition}
This proposition allows one to define by density, for any $U_0\in \overline{D(\cal{A})}$ and $F= F(t)\in L^{1}(0,T;{\cal{H}})$ a weak solution of \eqref{Eveq} such that $U(0) = U_0$, cf. H. Brezis \cite{BR1}.
\subsection{Functional setting}
Throughout this article we let $H$ and $V$ be two Hilbert spaces with norms respectively denoted by $\Vert.\Vert$ and $\vert.\vert$.  We assume that $V$ is densely and continuously embedded into $H$.
Identifying $H$ with its dual $H'$, we obtain $V\hookrightarrow H
=H' \hookrightarrow V'$. We denote inner products by (.,.) and duality products by $\langle \cdot ,\cdot \rangle$; the spaces in question will be specified by subscripts. The notation  $\langle f , u \rangle$ without any subscript will be used sometimes to denote $\langle f , u \rangle_{V', V}$. The duality map: $V\rightarrow V'$ will be denoted by $A$. We observe that $A$ is characterized by the property $$\forall (u,v)\in V\times V, \quad \langle Au, v \rangle_{V', V} = (u, v)_V.$$
\subsection{Weak solutions } 
We consider the dissipative evolution equation:
\begin{equation}\label{eq1}
 \ddot{u} + Au + g(\dot{u}) = h(t)
\end{equation} where $g\in C(V, V')$ satisfies 
\begin{equation}\label{mon} \forall ( v, w)\in V\times V,\quad \langle g(v)-g(w) , v-w \rangle \ge 0 .\end{equation} We consider the (generally unbounded) operator $\cal{A}$ defined on the Hilbert space ${\cal H} = V\times H$ by 
$$D({\cal{A}}) = \{ (u,v)\in V\times V,   \,\, Au+ g(v) \in H \} $$ and 
$$\forall (u,v)\in D({\cal{A}}) ,\quad  {\cal{A}} (u,v)= (-v,  Au+ g(v)). $$ 
 \begin{lemma}\label{mon-op} The operator $ {\cal{A}}$ is maximal monotone.

\end{lemma}
\begin{proof}  Let $U= (u,v)$ and $\hat{U}= (\hat{u},\hat{v})$ be two elements of $D({\cal{A}})$. We have 
$$ ({\cal{A}}U - {\cal{A}}\hat{U}, U-\hat{U})_ {\cal{H}} = - (u-\hat{u}, v-\hat{v})_V + (Au +g(v)-A\hat{u }-g(\hat{v}), v-\hat{v})_H$$
$$ - (u-\hat{u}, v-\hat{v})_V + \langle Au +g(v)-A\hat{u }-g(\hat{v}), v-\hat{v}\rangle_{V', V}$$  since $Au +g(v)\in H $ and $A\hat{u }+g(\hat{v})\in H$ while $v, \hat{v}$ are in V . This reduces to  $$ ({\cal{A}}U - {\cal{A}}\hat{U}, U-\hat{U})_ {\cal{H}} =  \langle g(v)-g(\hat{v}), v-\hat{v}\rangle_{V', V}\ge 0 $$ Hence  $ {\cal{A}}$ is monotone. To prove that $ {\cal{A}}$ is maximal monotone we are left to show that for any $F = (\varphi, \psi )\in {\cal{H}}$ the following equation 
$$ u-v=\varphi,  \quad Au +g(v) + v = \psi $$ has a solution $U= (u,v)\in D({\cal{A}})$.  This is  equivalent to finding a solution $v\in V$ of 
$$  \quad Av +g(v) + v = \psi - A\varphi \in V' $$ But now the operator ${\cal{C}}\in C(V, V')$  defined by 
$$ \forall v \in V,  \quad {\cal{C}}v = Av +g(v) + v $$ is continuous and coercive: $V \rightarrow V'. $ Therefore by  Corollary 14 p. 126 from H. Brezis \cite{BR2},  ${\cal{C}}$ is surjective. Finally $ {\cal{A}}$ is maximal monotone as claimed.\end{proof}

As a consequence of Proposition \ref{eveq}, for any $h\in L^{1}_{loc}(\mathbb{R^+}, H)$ and for each $(u_0, u_1)\in V\times H $ there is a unique weak solution $$u \in C(\R^{+},V)\cap C^{1}(\R^{+},H)$$ of \eqref{eq1}  such that $ u(0) = u_0$ and  $\dot{u}(0) = u_1$. This solution  can be recovered on each compact interval $[0, T]$ by approximating the initial data by elements of the domain, the forcing term $h$ by $C^1$ functions and passing to the limit: the limit is independent of the approximating elements so chosen. The next result shows that in fact the approximation can even be made uniform on $\mathbb{R^+}$.

\subsection{Density of strong solutions} \begin{lemma}\label{uniform-density} For any $h\in L^{2}_{loc}(\mathbb{R^+}, H)$ and for each $(u_0, u_1)\in V\times H $ and for each $\delta>0$ there exists $({w_0}, {w_1})\in D(\cal {A} )$ and ${k}\in C^{1}(\mathbb{R^+}, H)$ for which the solution $w\in  W^{1,1}_{loc}(\R^{+},V)\cap W^{2,1}_{loc}(\R^{+},H)$ of 
$$  \ddot{w} + Aw + g(\dot{w}) = k(t); \quad w(0) = w_0, \,\, \dot{w}(0) = w_1$$ satisfies
$$ \forall t\ge 0, \quad \Vert u(t)-w(t)\Vert + \vert \dot{u}(t)-\dot{w}(t)\vert \le \delta $$ and in addition $$\forall t\in \R^{+}, \quad \int _t^{t+1} \vert k(s)-h(s)\vert^2 ds \le 2\delta . $$ 
\end{lemma}
\begin{proof} It suffices to use the last result of Proposition \ref{eveq} by observing that for any $h\in L^{2}_{loc}(\mathbb{R^+}, H)$ we can find  ${k}\in C^{1}(\mathbb{R^+}, H)$ such that $$\forall n\in \mathbb{N}, \quad \int _n^{n+1} \vert k(s)-h(s)\vert ^2 ds \le \delta 2^{-2n-2}. $$ Choosing $({w_0}, {w_1})\in D(\cal {A} )$ such that $ \Vert w_0 -u_0\Vert +\vert w_1-v_1\vert  \le \delta 2^{-1}$ the result follows immediately \end{proof}

\section{A general ultimate bound.} In this section we give a quite different proof, in a slightly more general case,  of a result stated in \cite{HR}, Remark 1.2, b) p. 167. We assume that $h\in S^2 (\mathbb{R^+}, H)$ with $$ S^2(R^+, H) = \{ f\in L^2_{loc} (R^+, H),\quad \sup _{t\in R^+} \int_t^{t+1}\vert f(s)\vert^2 ds <\infty\} $$ and we set $$||h|| _{S^2 ( \mathbb{R^+}, H)} = \{\sup _{t\in R^+} \int_t^{t+1}\vert h(s)\vert^2 ds\}^{1/2} $$ In particular if $h\in L^{\infty}(\mathbb{R^+}, H)$ , then $h\in S^2 (\mathbb{R^+}, H)$ and $||h|| _{S^2 ( \mathbb{R^+}, H)}\le ||h|| _{ L^{\infty}(\mathbb{R^+}, H)}$.
\begin{thm}\label{theo0}
 Assume that  $g\in C(V, V')$ satisfies the condition \eqref{mon} and 

\begin{equation} \label{condi1}\exists \gamma>0,\,\,\exists C_{1}\geq0\quad \forall v\in V,
\quad \langle g(v), v \rangle \geq \gamma \vert v \vert ^{2} - C_{1}.\end{equation}

\begin{equation} \label{condi2}\exists K>0,\,\,\exists C_{2}\geq0\quad \forall v\in V,\quad \Vert g(v) \Vert_{V'} \leq  C_{2} + K\langle g(v), v\rangle. \end {equation}
Then any solution $u \in C(\R^{+},V)\cap C^{1}(\R^{+},H)$ of $\eqref{eq1}$ is bounded on $\R^{+}$ in the sense that $u$ has bounded range in $V$ and $\dot{u}$ has bounded range in $H$. In addition we have  for some constant $K$ depending only on $A$ and $g$
$$ \limsup_{t\rightarrow \infty} \, (\vert \dot{u}(t) \vert^2 + \Vert u(t) \Vert^2) \le K(1+ ||h||^4 _{S^2 (\mathbb{R^+}, H)})$$ 
\end{thm}
\begin{proof} The boundedness result is known for local damping operators $g$ , cf. the second case of Theorem IV.2.1.1 of \cite{H2} and in the general case it can be proved by adapting in this case the method from \cite {ABH^2}. However even in the local case these results cannot provide a reasonable estimate of the ultimate bound. We start by an estimate in the case of a strong solutions, i.e. we assume $u \in W^{1,1}_{loc}(\R^{+},V)\cap W^{2,1}_{loc}(\R^{+},H)$. The general case will follow by density. 
Let $$E(t) = \frac{1}{2}(\vert \dot{u} \vert^2 + \Vert u \Vert^2)$$ Under the regularity conditions $[u_{0}, v_{0}]\in V \times V$, $\gamma(0, v_{0})\in H$ and $h\in W^{1,1}_{loc}(\mathbb{R^+}, H),$\:
 $ t\rightarrow E(t)$ is absolutely continuous and we have $\forall t\in\mathbb{R^+}$:
\begin{eqnarray}\label{deriv0}
\frac{d}{dt}E(t) = (h,\dot{u}) - \langle g(\dot{u}), \dot{u}\rangle.
\end{eqnarray}

In addition  $t \rightarrow(u(t), \dot{u}(t))$ is absolutely continuous and $$\frac{d}{dt}(u(t), \dot{u}(t)) = \vert \dot{u} \vert^2 - \Vert u \Vert^2 - \langle g(\dot{u}), u  \rangle + (h,u).$$
 By using $\eqref{condi2}$, we obtain:\begin{eqnarray}\label{deriv1}\frac{d}{dt}(u(t), \dot{u}(t)) \leq \vert \dot{u} \vert^2 - \Vert u \Vert^2 +  \Vert u \Vert  \left\{ P\vert h \vert + C_{2} + K \langle g(\dot{u}), \dot{u}\rangle \right\}\end{eqnarray}
with $$P= \sup \{\vert u \vert ,\,\,  u\in V,\,\, \Vert u \Vert = 1\}.$$ Introducing  $$\Phi(t) = 2 E(t), \quad \forall t \geq 0, $$ we are reduced to estimate the upper limit bound for $\Phi(t)$.  Let us introduce 
$$ M = \limsup_{t\rightarrow \infty} \, \Phi(t) $$ and let us consider a sequence of times $t_n$ tending to infinity for which 
$$ \Phi(t_n) \ge M - 1/n $$ In addition for $n$ large enough we have $ t_n\ge \tau$ and
$$ \Phi(t_n-\tau) \le M +1/n$$ where $\tau$ is any fixed positive number to be chosen later. Therefore by integrating \eqref{deriv0} on $[t_n-\tau, t_n]$ we find 
$$ \int_{t_{n}-\tau}^{t_n}\langle g(\dot{u}), \dot{u}\rangle dt \le 1/n + \int_{t_{n}-\tau}^{t_n}(h,\dot{u}\rangle  dt \le 1/n + \frac{\gamma}{2}\int_{t_{n}-\tau}^{t_n} |\dot{u}|^2 dt + \frac{1}{2\gamma}\int_{t_{n}-\tau}^{t_n} |h|^2 dt $$ As a consequence of \eqref{condi1} we deduce 
\begin{equation}\label{diss}\int_{t_{n}-\tau}^{t_n}\langle g(\dot{u}), \dot{u}\rangle dt \le 2/n +  \frac{1}{\gamma}\int_{t_{n}-\tau}^{t_n} |h|^2 dt  + C_3 \end{equation}and \begin{equation}\label{kin} \int_{t_{n}-\tau}^{t_n} |\dot{u}|^2 dt  \le C_4/n +  C_5\int_{t_{n}-\tau}^{t_n} |h|^2 dt  + C_6 \end{equation} which provides an average bound of the kinetic part independent of the initial data and the transient behavior. This is remarkable since we only used the properties $ \Phi(t_n) \ge M - 1/n $ and $ \Phi(t_n-\tau) \le M +1/n $ to express the fact that $t$ is large. By combining these two estimates we also find an estimate of the form \begin{equation}\label{diff}|\Phi(t)- \Phi(s)| \le C_7(1+\int_{t_{n}-\tau}^{t_n} |h|^2 dt ) \end{equation} valid for all $s$, $t$ in $[t_n-\tau, t_n]$.  As a consequence if we had an $L^1$ estimate of the total energy instead of the kinetic part the proof would be completed with exponent $2$ instead of $4$. The difficulty in fact comes from the potential energy. From \eqref{deriv1}, by integrating on $[t_n-\tau, t_n]$ we find $$\int_{t_{n}-\tau}^{t_n}\Vert u \Vert^2 dt \leq  \int_{t_{n}-\tau}^{t_n}\vert \dot{u} \vert^2 dt +  \int_{t_{n}-\tau}^{t_n}[\Vert u \Vert  \left\{ P\vert h \vert + C_{2} + K \langle g(\dot{u}), \dot{u}\rangle \right\}]dt  + |[(u(t), \dot{u}(t))] _{t_n-\tau}^{t_n}|$$ Recalling the notation 
$ M = \limsup_{t\rightarrow \infty} \, \Phi(t) $ we find \begin{equation}\label{pot.}\int_{t_{n}-\tau}^{t_n}\Vert u \Vert^2 dt \leq  \int_{t_{n}-\tau}^{t_n}\vert \dot{u} \vert^2 dt +  M^{1/2} \int_{t_{n}-\tau}^{t_n}[\left\{ P\vert h \vert + C_{2} + K \langle g(\dot{u}), \dot{u}\rangle \right\}]dt  + C_8 M \end{equation} where $C_8 $ does not depend on $\tau$. Combining \eqref{kin} and \eqref{pot.} we obtain \begin{equation}\int_{t_{n}-\tau}^{t_n}\Phi(t)dt \le  C_5\int_{t_{n}-\tau}^{t_n} |h|^2 dt  + C_9 +  M^{1/2} \int_{t_{n}-\tau}^{t_n}[\left\{ P\vert h \vert + C_{2} + K \langle g(\dot{u}), \dot{u}\rangle \right\}]dt  + C_8 M \end{equation} and by \eqref{diss} this implies \begin{equation}\label{averenerg}\int_{t_{n}-\tau}^{t_n}\Phi(t)dt \le  C_{10} (1+ \int_{t_{n}-\tau}^{t_n} |h|^2 dt ) (1 +  M^{1/2})  + C_8 M \end{equation} Finally by combining this last inequality with \eqref{diff} we end up with \begin{equation}\label{final}(\tau -C_8)M \le  C_{11} (1+ \int_{t_{n}-\tau}^{t_n} |h|^2 dt ) (1 +  M^{1/2})  \end{equation} Fixing $\tau\ge 1+C_8$, the result now follows easily  since $$ \int_{t_{n}-\tau}^{t_n} |h|^2 dt\le (1+\tau) ||h||^2_{S^2}$$ The general case of weak solutions follows easily by density relying on Lemma \ref{uniform-density}.\end{proof}

\begin{remark} This ultimate bound has been obtained under the most general known assumption ensuring boundedness of trajectories. It seems not to depend on the kind of damping operator as long as the coerciveness and growth conditions are satisfied. We have absolutely no idea whether it has a chance to be optimal in some cases. A more natural quadratic estimate is valid in many cases as we shall see in the next section. \end{remark}

\section{The case of a power-like damping term} For the main result of this section, we need to introduce an additional Banach space Z such that $$ V\subset Z\subset H $$ with continuous imbeddings. The norm in $Z$ of a vector $z\in Z$ will be denoted by $||z||_Z$.  
\subsection{Main result}
\begin{thm}\label{power}Assume that  $g\in C(V, V')$ satisfies the condition \eqref{mon} and for some $\alpha\ge 0$ we have 
\begin{equation} \label{condi1bis}\exists \gamma>0,\,\,\exists C_{1}\geq0\quad \forall v\in V,
\quad \langle g(v), v \rangle \geq \gamma || v ||_Z^{\alpha +2} - C_{1}.\end{equation}
\begin{equation} \label{condi2bis}\exists K>0,\,\,\exists C_{2}\geq0\quad \forall v\in V,\quad \Vert g(v) \Vert_{V'} \leq  C_{2} + K ||v ||_Z^{\alpha +1}. \end {equation}
Then any solution $u \in C(\R^{+},V)\cap C^{1}(\R^{+},H)$ of $\eqref{eq1}$ is bounded on $\R^{+}$ in the sense that $u$ has bounded range in $V$ and $\dot{u}$ has bounded range in $H$. In addition we have  for some constant $K$ depending only on $A$ and $g$
$$ \limsup_{t\rightarrow \infty} \, (\vert \dot{u}(t) \vert^2 + \Vert u(t) \Vert^2) \le K(1+ ||h||^2 _{S^2 (\mathbb{R^+}, H)})$$ 
\end{thm}

\begin{proof} We start as in the proof  of Theorem \ref{theo0}:  by integrating \eqref{deriv0} on $[t_n-\tau, t_n]$ we find 
$$ \int_{t_{n}-\tau}^{t_n}\langle g(\dot{u}), \dot{u}\rangle dt \le 1/n + \int_{t_{n}-\tau}^{t_n}(h,\dot{u}\rangle  dt \le 1/n + \frac{\gamma}{2}\int_{t_{n}-\tau}^{t_n} ||\dot{u}||_Z^{\alpha +2} dt + C(\gamma) \int_{t_{n}-\tau}^{t_n} |h|^{\frac {\alpha +2}{\alpha +1}} dt $$ As a consequence of \eqref{condi1bis} we deduce , since $n\ge 1$, 
\begin{equation}\label{dissbis}\int_{t_{n}-\tau}^{t_n}\langle g(\dot{u}), \dot{u}\rangle dt \le  C(\gamma, \tau) [\int_{t_{n}-\tau}^{t_n} |h|^2 dt] ^{\frac {\alpha +2}{2\alpha +2}} + C_3 \end{equation}and \begin{equation}\label{kinbis} \int_{t_{n}-\tau}^{t_n} |\dot{u}|^2 dt  \le C_4 +  C_5(\gamma, \tau)\int_{t_{n}-\tau}^{t_n} |h|^2 dt \end{equation} From \eqref{dissbis} we also deduce the important new estimate
\begin{equation} \int_{t_{n}-\tau}^{t_n} ||\dot{u}||_Z^{\alpha+1} dt  \le C_6 +  C_7(\gamma, \tau)\{\int_{t_{n}-\tau}^{t_n} |h|^2 dt \}^{1/2}\end{equation} and by \eqref{condi2bis} this implies 
\begin{equation}\label{kinter} \int_{t_{n}-\tau}^{t_n} ||g(\dot{u})||_{V'}dt  \le C_8 +  C_9(\gamma, \tau)\{\int_{t_{n}-\tau}^{t_n} |h|^2 dt \}^{1/2}\end{equation}Recalling the notation 
$ M = \limsup_{t\rightarrow \infty} \, \Phi(t) $ we now find \begin{equation}\int_{t_{n}-\tau}^{t_n}\Vert u \Vert^2 dt \leq  \int_{t_{n}-\tau}^{t_n}\vert \dot{u} \vert^2 dt +  C_9(\gamma, \tau)M^{1/2}\{\int_{t_{n}-\tau}^{t_n} |h|^2 dt \}^{1/2} + C_{10} M + C_{11}\end{equation} where $C_{10 }$ does not depend on $\tau$. Then by using Cauchy-Schwarz 
\begin{equation}\label{potbis.}\int_{t_{n}-\tau}^{t_n}\Vert u \Vert^2 dt \leq  C_{12}(\gamma, \tau)\int_{t_{n}-\tau}^{t_n} |h|^2 dt + (C_{10}+1) M + C_{11}\end{equation} By choosing $\tau$ large enough we obtain, as a consequence of \eqref{potbis.} and \eqref{kinbis}, the inequality 
\begin{equation}\int_{t_{n}-\tau}^{t_n}\Phi(t)dt \le  C_{12}(\gamma)\int_{t_{n}-\tau}^{t_n} |h|^2 dt  + C_{13}\end{equation}
We conclude by using  
\begin{equation}\label{diffbis}|\Phi(t)- \Phi(s)| \le C_{14} (\gamma)(1+\int_{t_{n}-\tau}^{t_n} |h|^2 dt ) \end{equation} which is valid for all $s$, $t$ in $[t_n-\tau, t_n]$ and follows easily from \ref{dissbis} and \ref{kinbis}.
\end{proof}\bigskip \noindent
 
\begin{remark} This result is optimal. For instance if we consider an eigenvector $\varphi$ of $A$ corresponding to the eigenvalue $\lambda>0$, then for each $k>0$, $k \varphi$  is a stationary solution of the equation with source term $h(t) \equiv k \lambda \varphi$  for any dissipative operator $g$. This shows that the ultimate bound of the energy is at least quadratic with respect to the size of the source term.\end{remark} \newpage

\subsection{Examples} 
In this section, $\Omega$ denotes a bounded open domain of $\R^N$with $C^2$  boundary  and $\alpha\ge 0, \,\,c>0.$ We consider four simple special cases 

\noindent {\bf Example 1: The wave equation with local damping }

 \begin{equation}
\left\{\begin{array}{c}
 u_{tt}+c |{u_t}|^{\alpha} u_t - \Delta u = h(t,x),\ \hbox{ in }\R_{+}\times\Omega,\\[2mm]
 {u} = 0   \  \hbox { on } \R_{+}\times\partial\Omega.
\end{array}\right.
\end{equation} Here $V= H_0^1(\Omega) $,  $H= L^2(\Omega)$ and $Z = L^{\alpha+2}(\Omega).$ We assume $ (N-2)\alpha \le 2. $

\medskip

\noindent {\bf Example 2: The wave equation with nonlinear averaged damping }

 \begin{equation}
\left\{\begin{array}{c}
 u_{tt}+c [\int_{\Omega }{u_t}^2(t, x)dx]^{\frac{\alpha}{2}} u_t - \Delta u = h(t,x),\ \hbox{ in }\R_{+}\times\Omega,\\[2mm]
 {u} = 0   \  \hbox { on } \R_{+}\times\partial\Omega.
\end{array}\right.
\end{equation} Here $V= H_0^1(\Omega) $ and  $H= L^2(\Omega)= Z .$

\medskip

\noindent 
{\bf Example 3: A clamped plate equation with nonlinear structural averaged damping} 

\begin{equation}\label{delta2-D}
\left\{
\begin{array}{c}
 u_{tt}-c{[\int_{\Omega}\vert \nabla u_t\vert ^2 dx]}^{\frac{\alpha}{2}} \Delta u_t+\Delta^2 u = h(t, x),\ \hbox{ in }\R_{+}\times\Omega,\\[2mm]
 u=\vert \nabla u\vert = 0   \, \hbox { on } \R_{+}\times\partial\Omega.
\end{array}
\right.
\end{equation} Here $V= H_0^2(\Omega) $, $H= L^2(\Omega) $ and  $Z= H_0^1(\Omega).$

\medskip

\noindent 
{\bf Example 4: A simply supported plate equation with nonlinear structural averaged damping} 

\begin{equation}\label{delta2}
\left\{
\begin{array}{c}
 u_{tt}-c{[\int_{\Omega}\vert \nabla u_t\vert ^2 dx]}^{\frac{\alpha}{2}} \Delta u_t+\Delta^2 u = h(t, x),\ \hbox{ in }\R_{+}\times\Omega,\\[2mm]
 u= \Delta u = 0  , \  \hbox { on } \R_{+}\times\partial\Omega.
\end{array}
\right.
\end{equation} Here  $V= H^2\cap H_0^1(\Omega) $,  $H= L^2(\Omega)$and  $Z= H_0^1(\Omega).$

\medskip \noindent  As a consequence of  Theorem \ref{power} we obtain immediately \begin{corollary}\label{ex} In all the four examples, Let   $h\in S^2 (\mathbb{R^+}, H)$.Then any solution $u \in C(\R^{+},V)\cap C^{1}(\R^{+},H)$ of $\eqref{eq1}$ is bounded on $\R^{+}$ in the sense that $u$ has bounded range in $V$ and $\dot{u}$ has bounded range in $H$. In addition we have  for some constant $K$ independent of $h$ and the initial data $$ \limsup_{t\rightarrow \infty} \, (\vert \dot{u}(t) \vert^2 + \Vert u(t) \Vert^2) \le K(1+ ||h||^2 _{S^2 (\mathbb{R^+}, H)})$$  \end{corollary}

\bigskip \noindent
 
\begin{remark} In \cite{ABH^2}, for the 4 previous examples, the authors proved the existence of a unique almost periodic solution when $h$ is an $S^2$- almost periodic  source. In this case (in particular if $h$ is periodic) the ultimate bound coincides with the supremum of the energy of the almost periodic solution. Actually, if we try to estimate directly the periodic solution, some boundary (in time) term disappear but the main part of the estimate is not much simpler. In addition we know that the estimate is essentially optimal, only the multiplicative constants might be worked out if one wants a more precise inequality. \end{remark}

\section{Partial results in the  anti-periodic  case.} If  $g$ is odd and $h$ is $\tau$ -anti periodic, i.e. if we have $$ h(t+\tau) = -h(t) $$ the interesting solutions are the anti-periodic ones, cf.e.g. \cite{H4} for existence results. Since such solutions have mean-value $0$, the solution can be estimated through its time-derivative, and because the estimate of the derivative is generally much better, we can expect an improvement on the energy bound. \medskip 

\noindent This idea is perfectly valid if $H$ is finite-dimensional, since then $u$ and $\dot{u}$ belong to the same space, but otherwise we have a problem to reach the norm of $u$ in $V$. At the present time we do not know what happens if $\dim H = \infty$.  For the time being we can only prove the following partial results. 

\begin{proposition}\label{antip.}Assume that  $V= H$, $h\in C(\R, H)$ is $\tau$ -anti periodic, that $g\in C(H, H)$ satisfies the condition \eqref{mon} and for some $\alpha\ge 0$ we have 

\begin{equation} \label{condi1a}\exists \gamma>0,\,\,\exists C_{1}\geq0\quad \forall v\in H,
\quad \langle g(v), v \rangle \geq \gamma | v |^{\alpha +2} - C_{1}.\end{equation}

\begin{equation} \label{condi2a}\exists K>0,\,\,\exists C_{2}\geq0\quad \forall v\in H,\quad \vert g(v) \vert \leq  C_{2} + K |v |^{\alpha +1}. \end {equation}
Then any $\tau$ -antiperiodic solution $u $ of $\eqref{eq1}$ is such that
$$ \sup_{t\in\R} \, (\vert \dot{u}(t) \vert^2 + \vert u(t) \vert^2) \le C(1+ ||h||^{\frac{2 }{\alpha+1}}_{L^{\infty} (\R, H)})$$ where $C$ is independent of $h$. 
\end{proposition}

\begin{proof} The starting point is the same as for the proof of Theorem \ref{power}. From the inequality \begin{equation}\int_{0}^{2\tau} \langle g(\dot{u}), \dot{u}\rangle dt \le  C(\gamma, \tau) [\int_{0}^{2\tau}  |h|^2 dt] ^{\frac {\alpha +2}{2\alpha +2}} + C_3 \end{equation}we deduce the more precise estimate  \begin{equation}\label{kinanti} \int_{0}^{2\tau}  |\dot{u}|^2 dt  \le C_4 +  C_5(\gamma, \tau)[\int_{0}^{2\tau}  |h|^2 dt]^{\frac {1}{\alpha +1}} \end{equation} which implies, since $u$ has mean-value $0$, 
\begin{equation}\label{supu} \sup_{t\in [0, 2\tau]} \,  \vert u(t) \vert^2 \le C(1+ ||h||^{\frac{2 }{\alpha+2}}_{L^{\infty} (\R, H)})\end{equation} To obtain the uniform bound on $\dot{u}$, the trick now consists in evaluating the maximum of 
$ \Phi(t) = |\dot{u}|^2 + |A^{1/2}u |^2 . $ At a maximum point $ \theta$ the derivative vanishes, which gives 
$$  \langle g(\dot{u}), \dot{u}\rangle = (h,\dot{u}) $$ hence 
$$ |\dot{u}(\theta)|^2 \le C' (1+ |h|^{\frac {2}{\alpha +1}}) $$ This implies $$ \max_{t\in [0, 2\tau]}\Phi (t) =   \Phi(\theta) \le C'' (1+ ||h||^{\frac {2}{\alpha +1}}_{L^{\infty} (\R, H)})) $$ and the conclusion follows immediately. 
\end{proof}
\begin{remark} This result is optimal. For instance if we consider an eigenvector $\varphi$ of $A$ corresponding to the eigenvalue $\lambda>0$, then for each $k>0$, $u_k (t) = k \cos(\lambda^{1/2} t)\varphi$  is a solution of the equation 
$$ \ddot{u}+ Au + g(\dot{u} ) = g(- k\lambda^{1/2} \sin(\lambda^{1/2} t)\varphi) =: h(t) $$  and the $L^{\infty} $ norm of the source term is less than a constant times $k^{\alpha+1} $  for $k$ large. Both $u$ and $h$ are anti-periodic. \end{remark} 

We have a weaker  result  (intermediate between Theorem \ref{power}  and Proposition \ref{antip.}) which is also valid in the infinite dimensional setting and can be stated as follows: 
\begin{proposition}\label{antip2.} Assume that the conditions of Theorem \ref{power} are satisfied with \eqref{condi2bis} reinforced into \begin{equation} \label{condi2ter}\exists K>0,\,\,\exists C_{2}\geq0\quad \forall v\in V,\quad \Vert g(v) \Vert_{Z'} \leq  C_{2} + K ||v ||_Z^{\alpha +1}. \end {equation} 
Then any $\tau$ -antiperiodic solution $u \in C^1(\R, V)\cap C^2(\R, H)$ of $\eqref{eq1}$ is such that
$$ \sup_{t\in\R} \, (\vert \dot{u}(t) \vert^2 + \Vert u(t) \Vert^2) \le C(1+ ||h||^{\frac{\alpha +2 }{\alpha+1}}_{L^{2} ([0, \tau], H)})$$ where $C$ is independent of $h$. 
\end{proposition}

\begin{proof} The starting point is the same as for the proof of Theorem \ref{power}. From the inequality \begin{equation}\int_{0}^{2\tau} \langle g(\dot{u}), \dot{u}\rangle dt \le  C_{3} [ 1+ \int_{0}^{2\tau}  |h|^2 dt] ^{\frac {\alpha +2}{2\alpha +2}}] \end{equation}we deduce the  estimate  \begin{equation}\label{kinanti2} \int_{0}^{2\tau}  |\dot{u}|^2 dt  \le C_4 [ 1+  \{\int_{0}^{2\tau}  |h|^2 dt\}^{\frac {1}{\alpha +1}}] \end{equation} but also 
\begin{equation}\label{Zest} \int_{0}^{2\tau} ||\dot{u}||_Z^{\alpha+1} dt  \le C_5 [ 1+ \{\int_{0}^{2\tau} |h|^2 dt \}^{1/2}]\end{equation} and by \eqref{condi2ter} this implies 
\begin{equation}\label{kinter} \int_{0}^{2\tau} ||g(\dot{u})||_{Z'}dt  \le C_6 [  1+ \{\int_{0}^{2\tau} |h|^2 dt \}^{1/2}]\end{equation}
From \eqref{Zest}, since $u$ has mean-value $0$, we deduce 
\begin{equation}\label{supu} \sup_{t\in [0, 2\tau]} \,  \Vert u(t) \Vert _Z \le C_7(1+ ||h||^{\frac{1 }{\alpha+1}}_{L^{2} ([0,2\tau],  H)})\end{equation} The two last inequalities imply immediately $$ \vert \int_{0}^{2\tau} \langle g(\dot{u}), u\rangle dt  \vert \le C_8 (1+ ||h||^{\frac{\alpha+2 }{\alpha+1}}_{L^{2} ([0,2\tau],  H)})$$  Now, multiplying the equation by $u$ and integrating on the period we find easily after combining with \eqref{kinanti2}
$$  \int_{0}^{2\tau} \Phi(t) dt   \le C_9 (1+ ||h||^{\frac{\alpha+2 }{\alpha+1}}_{L^{2} ([0,2\tau],  H)})$$ with $\Phi(t) = \vert \dot{u} \vert^2 + \Vert u \Vert^2 $ . Since 
$$ \Phi '(t) = (h, \dot{u}) - \langle g( \dot{u}), \dot{u}\rangle $$  by $2\tau$- periodicity and the inequality  $ \Phi'\le \vert h\vert \vert \dot{u}\vert + C_1$, we find  as a consequence of \eqref{kinanti2} $$ \Phi(t) \le \int _{t-\tau}^t \Phi(s) ds +  C_{10} (1+ ||h||^{\frac{\alpha+2 }{\alpha+1}}_{L^{2} ([0,2\tau],  H)})$$  and the conclusion follows easily by using $\tau-$ antiperiodicity. 
\end{proof}
\begin{remark} This result is certainly not optimal but it is all we can prove for the moment even in the most basic examples. Our result  requires additional regularity on $u$, this is usually achieved by assuming some regularity on $h$. When $g$ is monotone, usually the anti-periodic solution is unique and depends continuously on $h$ in $L^2$, so that the estimate will be easy to transfer to the general case in the examples. This is important since we cannot derive strong estimates on solutions which are not anti-periodic and therefore approximation by strong solutions has to be performed within the anti-periodic class. \end{remark}

\end{document}